\documentclass[a4paper, 10pt, conference]{ieeeconf}      

\IEEEoverridecommandlockouts                              

\usepackage{amsmath} 
\usepackage{amssymb}  

\title{\LARGE \bf
Partial  Realization Theory and System Identification Redux
}

\author{Anders Lindquist$^{1}$
\thanks{$^{1}$Anders Lindquist is with Department of Automation, Shanghai Jiao Tong University, China, and the Royal Institute of Technology, Stockholm, Sweden
        {\tt\small alq@kth.se}}%
}

\newtheorem{theorem}{Theorem}

\newtheorem{proposition}[theorem]{Proposition}
\newtheorem{lemma}[theorem]{Lemma}

\begin{document}

\maketitle
\thispagestyle{empty}
\pagestyle{empty}

\begin{abstract}

Some twenty years ago 
we introduced a nonstandard matrix Riccati equation to solve the partial stochastic realization problem. In this paper we provide a new derivation of this equation in the context of system identification. This allows us to show that the nonstandard matrix Riccati equation is universal in the sense that it can be used to solve more general analytic interpolation problems by only changing certain parameters. Such interpolation problems are ubiquitous in systems and control. In this context we also discuss a question posed by R.E. Kalman in beginning of the 1970s.

\end{abstract}

\section{INTRODUCTION}

A classical basic problem in system identification is to estimate the (unknown) constant coefficient matrices $A,B,C,D$ in a stable linear (SISO) stochastic system 
\begin{equation}\label{syst}
 \begin{cases}
    x(t+1)=Ax(t)+Bu(t)\\
    \phantom{+1)} y(t)=Cx(t)+Du(t) 
 \end{cases},
\end{equation}
driven by white noise $\{ u(t)\}_{t\in\Bbb{Z}}$, from a record  of observations
\begin{equation}
\label{record}
y_0,y_1,y_2,\dots,y_N
\end{equation}
of the output process $\{ y(t)\}_{t\in\Bbb{Z}}$, which is stationary in steady state. Modulo possible unobservable and/or unreachable modes, choice of coordinates and placement of zeros, this is equivalent to finding a shaping filter 
\[
\text{white noise}\stackrel{u}\negmedspace{\longrightarrow}\fbox{$w(z)$}\negmedspace\stackrel{y}
{\longrightarrow} 
\]
with a minimum-phase transfer function
\begin{equation}
\label{w}
w(z)=\rho\frac{\sigma(z)}{a(z)},
\end{equation}
where 
\begin{align}
    \sigma(z)&=z^n+\sigma_1z^{n-1}+\dots +\sigma_n \label{sigma}  \\
    a(z)&=z^n+a_1z^{n-1}+\dots +a_n \label{a}
\end{align}
are {\em Schur polynomials}, i.e., polynomials with all its roots in the open unit disc, and $\rho$ is a positive number. Then the rational function 
\begin{equation}
\label{f}
f(z)=\frac12\frac{b(z)}{a(z)}
\end{equation}
satisfying 
\begin{equation}
\label{Ref}
\text{Re}\{f(e^{i\theta})\}=|w(e^{i\theta})|^2
\end{equation} 
is positive real (see, e.g., \cite{LPbook}), and a simple calculation shows that
\begin{equation}
\label{b}
b(z)=z^n+b_1z^{n-1}+\dots +b_n
\end{equation} 
is a Schur polynomial whose coefficients can be determined from the linear system of equations corresponding to the relation
\begin{equation}
\label{absigma}
a(z)b(z^{-1})+b(z)a(z^{-1})=2\rho^2\sigma(z)\sigma(z^{-1}).
\end{equation}
In fact, \eqref{Ref} is equivalent to 
\begin{equation}
\label{fwPhi}
f(z)+f(z^{-1})=\rho^2w(z)w(z^{-1}) =:\Phi(z),
\end{equation} 
where 
\begin{equation}
\label{Phi}
\Phi(z)=\sum_{k=-\infty}^\infty c_k z^{-k}
\end{equation}
is the power spectral density of the stationary output process $y$ and 
\begin{displaymath}
c_k =\mathbb{E}\{y(t+k)y(t)\}
\end{displaymath}
are the covariance lags \cite{LPbook}. Then, 
\begin{equation}
\label{fexpansion}
f(z)=\tfrac12 + c_1z^{-1}+ c_2z^{-2}+ \dots
\end{equation}
is analytic in the complement of the unit disc in the complex plane and maps to the right half plane, estabishing the positive-real property.

Now, if we had an infinite observation record \eqref{record}, i.e, $N=\infty$, then we would have an infinite sequence $(c_0,c_1,c_2,\dots)$ of covariance lags that could be determined from the ergodic limits
\begin{displaymath}
c_k = \lim_{N\to\infty}\frac{1}{N+1}\sum_{t=0}^{N-k}y_{t+k}y_t ;
\end{displaymath}
see, e.g., \cite{LPbook}. Then, identifying coefficients of powers of $z$ in $b(z)=2f(z)a(z)$ as done in \cite{BLpartial97}, we obtain
\begin{equation}\label{c2b}
\begin{bmatrix}
b_1\\b_2\\ \vdots\\b_d
\end{bmatrix}
=2\begin{bmatrix}
c_1\\c_2\\ \vdots\\c_d
\end{bmatrix}
+
\begin{bmatrix}
1&&&\\2c_1&1&&\\ \vdots&\vdots&&\\
2c_{d-1}&2c_{d-2}&\dots&1
\end{bmatrix}
\begin{bmatrix}
a_1\\a_2\\ \vdots\\a_d
\end{bmatrix}
\end{equation}
for  nonnegative powers and 
\begin{equation}\label{Hankel} 
\begin{bmatrix} c_1& c_2& \cdots &c_{d}\\ 
c_2& c_3 &\cdots & c_{d+1}\\
\vdots &\vdots& \ddots& \vdots\\
 c_{d}&c_{d+1}&\cdots&c_{2d-1}\end{bmatrix}
\begin{bmatrix}
a_1\\a_2\\ \vdots\\a_d
\end{bmatrix}=
-\begin{bmatrix}
c_{d+1}\\c_{n+2}\\ \vdots\\c_{2d}
\end{bmatrix}
\end{equation}
for negative powers. The coefficient matrix in \eqref{Hankel} is a Hankel matrix that we denote $H_d$. By Kronecker's theorem \cite{Kalmanbook}, 
\begin{equation}
\label{Kronecker}
d:= \deg f(z) =\text{rank}\, H_\infty = \text{rank}\, H_d.
\end{equation}
Hence, $f(z)$, and hence also $w(z)$, can be determined from finite sequence $(c_0,c_1,\dots,c_n)$ of covariance lags, where  
$n:=2d$.  

Therefore it may seem that only a finite covariance record $(c_0,c_1,\dots,c_n)$ is needed, and this also a basic assumption in the early work on so-called {\em subspace identification} \cite{Aoki,OverscheeDeMoor93,OverscheeDeMoor96}, where in general the biased ergodic estimates 
\begin{displaymath}
c_k=\frac{1}{N+1}\sum_{t=0}^{N-k}y_{t+k}y_t 
\end{displaymath}
were used to insure that the corresponding Toeplitz matrix is positive definite, as required. However, as we pointed out in \cite{LP96}, this is incorrect and may lead to an $f(z)$ that is not positive real; also see \cite[Chapter 13]{LPbook}. This is due to the possible difference between {\em algebraic degree\/} and {\em positive degree\/} (Section~\ref{sec:degree}). The classical subspace identification procedures were based on solving the {\em deterministic partial realization problem\/} \cite{Kalmanbook,GraggL} rather than the stochastic one, namely the {\em rational covariance extension problem}, which we shall return to in Section~\ref{sec:background}. 

The focus of this paper will be on a certain nonstandard Riccati equation, called the {\em Covariance Extension Equation (CEE)}, which solves the rational covariance extension problem. It will be introduced in Section~\ref{sec:CEE}. Section~\ref{sec:background} as a whole is devoted to background material for the paper. In addition, in Section~\ref{sec:Kalman} we consider a question posed by Kalman and give a partial answer. In Section~\ref{sec:new} we provide a new derivation of CEE that will allow us to establish that versions of CEE can be used to solve more general analytic interpolation problems encountered in systems and control. Such an example will be given in Section~\ref{sec:NP}. We conclude with Section~\ref{sec:conclusions}, where we some future directions of research are discussed.

\section{BACKGROUND}\label{sec:background}

\subsection{Rational covariance Extension problem}

The rational covariance extension problem, first formulated by Kalman in \cite{KalmanToeplitz}, can be stated in the following way. Given a positive covariance sequence $(c_0,c_1,\dots,c_n)$, i.e., a sequence $(c_0,c_1,\dots,c_n)$ with the property that the Toeplitz matrix
\begin{displaymath}
T=\begin{bmatrix}c_0&c_1&\cdots&c_n\\c_1&c_0&\cdots&c_{n-1}\\ \vdots&\vdots&\ddots&\vdots\\c_n&c_{n-1}&\cdots&c_0
\end{bmatrix}
\end{displaymath}
is positive definite.
find an infinite extension $c_{n+1},c_{n+2},\dots$ such that the function \eqref{fexpansion} is a rational positive-real function of degree at most $n$. This is the proper partial realization problem connected to the system identification problem posed in the introduction. If we ignore the condition that $f(z)$ be positive real, we have a deterministic partial realization problem equivalent to Pad\'e approximation \cite{GraggL}. This is the problem solved in the original papers on subspace identification. 

The following theorem provides a smooth complete parameterization of the set set of solutions to rational covariance extension problem. 

\begin{theorem}\label{thm:covext}
Let $(c_0,c_1,\dots,c_n)$ be a positive covariance sequence. Then, given any Schur polynomial \eqref{sigma}, there is one and only one Schur polynomial \eqref{a} and $\rho>0$ such that 
\begin{displaymath}
w(z)=\rho\frac{\sigma(z)}{a(z)}
\end{displaymath}
 is a shaping filter for $(c_0,c_1,\dots,c_n)$. The mapping from $\sigma$ to $(a,\rho)$ is diffeomorphism. 
\end{theorem}

The existence part of Theorem~\ref{thm:covext} was proved in \cite{GeorgiouThesis} (also see \cite{Georgiou87}) and the rest of the theorem in \cite{BLGM}. Note that $\sigma(z)$ and $a(z)$ may have common roots, so the degree of $w(z)$ might be less than $n$.

For each parameter $\sigma$ there is a convex optimization problem solving for $(a,\rho)$, which first appeared in \cite{BGuL} (also see \cite{SIGEST,BEL}), but here we shall consider a different method of solution described in Section~\ref{sec:CEE}. 

\subsection{Algebraic and positive degree}\label{sec:degree}

The {\em algebraic degree\/} of $(c_0,c_1,\dots,c_n)$ is the minimal degree of a deterministic partial realization of $(c_0,c_1,\dots,c_n)$. It is related to the rank of a Hankel matrix \cite{GraggL} and has the generic value $[\tfrac{n}{2}]$, i.e., sequences  $(c_0,c_1,\dots,c_n)$ that fail to have this algebraic degree live on a thin (lower-dimensional) subset, and there is no nonempty open set of vectors $c=(c_0,c_1,\dots,c_n)'\in\mathbb{R}^{n+1}$ having an algebraic degree different from $[\tfrac{n}{2}]$.

The {\em positive degree\/} of $(c_0,c_1,\dots,c_n)$ is the minimum degree of any solution to the corresponding rational covariance extension problem. By contrast, it does not have a generic value, as seen from the following theorem proved in \cite{BLpartial97}. 

\begin{theorem}\label{thm:posdeg}
 For each $p$ such that $[n/2]\leq p\leq n$, there is a nonempty open set in $\mathbb{R}^{n+1}$ of sequences  $(c_0,c_1,\dots,c_n)$ for which $p$ is the positive degree. The maximal positive degree is $n$.
\end{theorem}

Since the original subspace identification algoritms are based on Hankel factorization, they produce solutions with the algebraic degree, which may not coincide with the required positive degree. In \cite{DLM98} we presented a system that produces data for which there is massive failure of the basic subspace identification algorithms. This led to quite a lot of work devising {\em ad hoc\/} fixes that are now included in the codes.

\subsection{Covariance extension equation}\label{sec:CEE}

Given the Schur polynomial \eqref{sigma},
we define 
\begin{equation}\label{sigmaGammah}
\sigma = \bmatrix \sigma_1\\ \sigma_2\\ \vdots\\ \sigma_n \endbmatrix,
\quad 
\Gamma  = \begin{bmatrix} -\sigma_1  & 1 & 0 & \cdots & 0\\ -\sigma_2 
& 0 & 1 & \cdots & 0\\
\vdots& \vdots & \vdots & \ddots & \vdots\\ -\sigma_{n-1} & 0 & 0 &
\cdots & 1\\ -\sigma_n & 0 & 0 & \cdots & 0 
\end{bmatrix} ,
\quad
h = \bmatrix 1\\0\\Ê\vdots\\0 \endbmatrix. 
\end{equation}Moreover, we represent the covariance data in terms of the first $n$ coefficients in the expansion 
\begin{equation}\label{expansion}
\begin{split}
&\frac{z^n}{z^n + c_1 z^{n-1}+\dots + c_n}\\ 
&\quad= 1 - u_1 z^{-1} - u_2z^{-2} - u_3 z^{-3} - \dots
\end{split}
\end{equation}
about infinity and define 
\begin{equation}\label{Uu}
u = \begin{bmatrix} u_1\\u_2\\ \vdots  \\u_n \end{bmatrix} ,
\qquad 
U  =  \begin{bmatrix} 0&  &   &  & \\ u_1&0 & & &\\ u_2&u_1& & & \\
\vdots &\vdots& \ddots&  &\\
u_{n-1}&u_{n-2}&\cdots&u_1&0\end{bmatrix}.
\end{equation}  
Finally define the the function $g: \Bbb{R}^{n\times n}\to \Bbb{R}^n$given by
\begin{equation}\label{g(P)} 
g(P)= u +U\sigma + U\Gamma Ph .
\end{equation}
Then the following nonstandard Riccati equation  
\begin{equation} \label{CEE}
 P = \Gamma (P-Phh'P) \Gamma' + g(P)g(P)' ,
\end{equation} 
where $^\prime$ denotes transposition, was called the Covariance Extension Equation (CEE) in  \cite{BLpartial97}.
It parameterizes the solutions to the rational covariance
extension problem in terms of the covariance data and the numerator polynomial $\sigma(z)$ corresponding to desired spectral zeros. In fact, the following theorem was proved in \cite{BLpartial97}.

\begin{theorem}\label{thm:CEE}
Let $(c_0,c_1,\dots,c_n)$ be a positive  covariance sequence. Then, for each Schur polynomial \eqref{sigma}, there is a unique symmetric  solution $P$ of CEE satisfying $h'Ph<1$. Moreover, for each $\sigma$ there is a unique shaping filter \eqref{w}  for $(c_0,c_1,\dots,c_n)$,
where $a(z)$ and $\rho$ are given in terms of the corresponding $P$ by
\begin{equation}
\label{arho}
\begin{split}
&a=(I-U)(\Gamma Ph+\sigma)-u,\\
&\rho=\sqrt{1-h'Ph} .
\end{split}
\end{equation} 
\end{theorem}
Here $a:=(a_1,a_2,\dots,a_n)'$. The degree of $w(z)$ equals the rank of $P$.

In \cite{BFL,BFLcdc05} we presented a homotopy continuation algorithm for solving CEE. 

\subsection{Kalman's question}\label{sec:Kalman}

In view of the fact that the algebraic degree can be determined from the rank of a Hankel matrix of the covariance data \cite{Kalmanbook,GraggL}, in 1972 Kalman \cite{Kalmanprivate} posed the question whether there is a similar matrix-rank criterion for determining the positive degree. In view of Theorem~\ref{thm:posdeg}, this would seem impossible. The closest we have found in this direction is the following result, which follows from Theorem~\ref{thm:CEE} and is  reported in \cite{BLpartial97}.

\begin{proposition}
Let $P(\sigma)$ be the unique solution of \eqref{CEE}.
Then the positive degree of the covariance sequence $(c_0,c_1,\dots,c_n)$ equals the minimum of $\text{rank}\, P(\sigma)$ over all Schur polynomials $\sigma(z)$.
\end{proposition}

\section{A NEW DERIVATION OF THE COVARIANCE EXTENSION EQUATION}\label{sec:new}

We present a new derivation of CEE which is divided into two separate steps, namely one that just imposes the condition that $f(z)$ be positive real and rational of degree at most $n$, and a second one imposing the interpolation condition that $f(z)$ should match the $n+1$ first covariances  $(c_0,c_1,\dots,c_n)$. In this way, we see that this nonstandard matrix Riccati equation is universal in the sense that it can be used to solve more general analytic interpolation problems only changing certain parameters. In fact, the first step remains the same in this more general context. 

\subsection{Stochastic realization}\label{sec:newA}

The rational function $f(z)$ defined by \eqref{f} has a realization
\begin{equation}
\label{freal}
f(z)=\tfrac12 + h'(zI-F)^{-1}g,
\end{equation}
where 
\begin{equation}
\label{F}
F= \begin{bmatrix} -a_1  & 1 & 0 & \cdots & 0\\
-a_2 & 0 & 1 & \cdots & 0\\
\vdots& \vdots & \vdots & \ddots & \vdots\\
-a_{n-1} & 0 & 0 &\cdots & 1\\ 
-a_n & 0 & 0 & \cdots & 0 
\end{bmatrix}= J-ah' ,
\end{equation}
$J$ is the upward shift matrix, and $g$ is an $n$-vector to be determined. Note that this need not be a minimal realization, as there could be cancellations of common zeros of $a(z)$ and $b(z)$. 

\begin{lemma}\label{lem:ab2g}
The vector $g$ in \eqref{freal} is given by 
\begin{equation}
\label{ab2g}
g=\frac12 (b-a),
\end{equation}
where  $a:=(a_1,a_2,\dots,a_n)'$ and $b:=(b_1,b_2,\dots,b_n)'$ are the $n$-vectors of coefficients in the polynomials $a(z)$ and $b(z)$, respectively, in \eqref{f}.
\end{lemma}

\begin{proof}
From \eqref{f} and \eqref{freal} we have
\begin{displaymath}
\frac{b(z)}{a(z)}=1 +2h'(zI-F)^{-1}g,
\end{displaymath}
to which we apply the matrix inversion lemma (Appendix) to obtain
\begin{displaymath}
\begin{split}
\frac{a(z)}{b(z)}&=1 -2h'(2gh'+zI-F)^{-1}g\\
&=1 -2h'[zI-(J-ah' -2gh')]^{-1}g.
\end{split}
\end{displaymath}
Hence, since $b(z)$ is the denominator polynomial, we must have $a+2g=b$, from which \eqref{ab2g} follows.
\end{proof}

In the same way, $w(z)$, given by \eqref{w}, has a realization
\begin{equation}
\label{wreal}
w(z)=\rho + h'(zI-F)^{-1}k
\end{equation}
for some $n$-vector $k$. 
Since $w(z)$ is a minimum-phase spectral factor of $f(z)+f(z^{-1})$, from stochastic realization theory \cite[Chapter 6]{LPbook} we have that
\begin{equation}
\label{rhok}
\rho=\sqrt{1-h'Ph}, \qquad k=\rho^{-1}(g-FPh),
\end{equation}
where $P$ is symmetric minimum solution of the algebraic Riccati equation
\begin{equation}
\label{Riccati}
P=FPF'+(g-FPh)(1-h'Ph)^{-1}(g-FPh)' .
\end{equation}

\begin{lemma}\label{lem:P2g}
The vectors $g$ and $k$ in \eqref{freal} and \eqref{wreal} are given by
\begin{subequations}\label{gk}
\begin{align}
   g &=\Gamma Ph+\sigma -a   \label{P2g}\\
   k & =\rho(\sigma -a), \label{P2k}
\end{align}
\end{subequations}
where $P$ is the minimal solution of \eqref{Riccati}.
\end{lemma}

\begin{proof}
Applying the matrix inversion lemma (Appendix) to 
\begin{displaymath}
\frac{\sigma(z)}{a(z)}=1+h'(zI-F)^{-1}\rho^{-1}k 
\end{displaymath}
yields
\begin{displaymath}
\begin{split}
\frac{a(z)}{\sigma(z)}&= 1-h'(\rho^{-1}kh' +zI-F)^{-1}\rho^{-1}k\\
&= 1-h'[zI-J+(a-\rho^{-1}k)h']^{-1}\rho^{-1}k.
\end{split}
\end{displaymath}
However, since the denominator is $\sigma(z)$, we must have 
\begin{displaymath}
J -(a-\rho^{-1}k)h'=J-\sigma h'=\Gamma,
\end{displaymath}
and hence \eqref{P2k} follows. Moreover, $\Gamma =F-\rho^{-1}k$. This together with \eqref{P2k} yields
\begin{displaymath}
(1-h'Ph)(\sigma -a)=g -JPh +ah'Ph,
\end{displaymath}
from which \eqref{P2g} follows.
\end{proof}

\begin{lemma}
The minimal solution of \eqref{Riccati} is also the minimal solution of 
\begin{equation}
\label{CEEprel}
P = \Gamma (P-Phh'P) \Gamma' +gg' .
\end{equation}
\end{lemma}

\begin{proof}
 Note that \eqref{Riccati} can be written
\begin{displaymath}
P=(\Gamma + \rho^{-1}kh')P(\Gamma +\rho^{-1}kh')'+kk',
\end{displaymath}
from which we have
\begin{displaymath}
\begin{split}
P-\Gamma P\Gamma' &= \rho^{-2}kk' +\rho^{-1}\Gamma Phk' +\rho^{-1}kh'P\Gamma'\\
&=(\Gamma Ph+\rho^{-1}k)(\Gamma Ph+\rho^{-1}k)' -\Gamma Phh'P\Gamma' ,
\end{split}
\end{displaymath}
which, in view of  \eqref{gk}, in turn yields \eqref{CEEprel}.
\end{proof}

\subsection{Interpolation condition}

Next we introduce the interpolation condition that $f(z)$ matches the first $n+1$ covariance lags $c_0,c_1,\dots,c_n$. To this end, we identify coefficients of nonnegative powers in $b(z)=2f(z)a(z)$ with $f(z)$ given by \eqref{fexpansion}. This yields 
\begin{equation}
\label{interpolation}
b=2c+(2C-I)a,
\end{equation}
where
\begin{equation}
\label{C}
C=\begin{bmatrix}1&&&&\\c_1&1&&&\\c_2&c_1&1&&\\ \vdots&\vdots&\vdots&\ddots&\\c_{n-1}&c_{n-2}&c_{n-3}&\cdots&1\end{bmatrix}.
\end{equation}
However, by Lemma~\ref{lem:ab2g}, $b=a+2g$, which inserted into \eqref{interpolation} yields $g+a=c+Ca$. Combining this with \eqref{P2g}  we have $a=C^{-1}(\Gamma Ph+\sigma-c)$, and hence \eqref{P2g} can be written
\begin{equation}
\label{gC}
g=C^{-1}c + (I-C^{-1})(\Gamma Ph+\sigma).
\end{equation}
Note that here $C$ is an $n\times n$ matrix and not a $d\times d$ matrix with $d$ being the algebraic degree of $c_0,c_1,\dots,c_n$ as in \eqref{c2b}. By Theorem~\ref{thm:posdeg}, $n$ is an upper bound of the positive degree.

Finally we show that \eqref{gC} is equivalent to \eqref{g(P)}. To this end, we first identify negative powers of $z$ in 
\begin{displaymath}
(1+c_1z^{-1}+\dots +c_n z^{-n})(1-u_1z^{-1}-u_2z^{-2}-\dots)=1,
\end{displaymath}
obtained from \eqref{expansion}, to obtain
\begin{displaymath}
c_k=u_k+\sum_{j=1}^{k-1}c_{k-j}u_j, \quad k=1,2,\dots,n,
\end{displaymath}
from which we have
\begin{equation}
\label{C2U}
Cu=c, \qquad C(I-U)=I .
\end{equation}
Consequently, $C^{-1}c=u$ and $I-C^{-1}=U$, so  \eqref{g(P)} follows from \eqref{g(P)}.

\section{GENERALIZATION TO ANALYTIC INTERPOLATION}\label{sec:NP}

Next we  show that the nonstandard Riccati equation \eqref{CEE} is universal in the sense that it holds for more general analytic interpolation problems by merely redefining the parameters $u$ and $U$. We shall demonstrate this for a Nevanlinna-Pick interpolation problem with rationality constraints. 

Given distinct point $z_0,z_1,\dots,z_n$ in the complement of the unit disc of the complex plane and points $c_0,c_1,\dots,c_n$ in the open right half-plane, find a rational positive real function $f(z)$ of degree at most $n$ satisfying the interpolation condition
\begin{equation}
\label{NP}
f(z_k)=c_k, \quad k=0,1,\dots,n.
\end{equation}
It is convenient to chose the points in conjugate pairs to ensure that $f$ is a real function. 

Clearly all calculations in Section~\ref{sec:newA} remain intact, and it on remains to enforce the interpolation condition, which we may write as 
\begin{displaymath}
b(z_k)=\frac12 c_ka(z_k) 
\end{displaymath}
or, equivalently, as
\begin{subequations}
\begin{equation}
\label{interpolationNP}
V\begin{bmatrix}1\\b\end{bmatrix}=\frac12 CV\begin{bmatrix}1\\a\end{bmatrix},
\end{equation}
where $V$ is the Vandermonde matrix
\begin{equation}
\label{Vandermonde}
V=\begin{bmatrix}z_0^n&z_0^{n-1}&\cdots&1\\z_1^n&z_1^{n-1}&\cdots&1\\\vdots&\vdots&&\vdots\\z_n^n&z_n^{n-1}&\cdots&1
\end{bmatrix}
\end{equation}
and $C$ is now the diagonal matrix
\begin{equation}
\label{Cdiag}
C=\begin{bmatrix}
c_0&&&\\&c_1&&\\&&\ddots&\\&&&c_n\end{bmatrix}.
\end{equation}
\end{subequations}
Since the points $z_0,z_1,\dots,z_n$ are distinct, the Vandermonde matrix \eqref{Vandermonde} is nonsingular, and hence we have
\begin{displaymath}
\begin{bmatrix}1\\b\end{bmatrix}=\frac12 V^{-1}CV\begin{bmatrix}1\\a\end{bmatrix}
\end{displaymath}
Therefore, by Lemma~\ref{lem:ab2g}, 
\begin{subequations}
\begin{equation}
\label{a2gNP}
\begin{bmatrix}0\\g\end{bmatrix}=T\begin{bmatrix}1\\a\end{bmatrix},
\end{equation}
where
\begin{equation}
\label{T}
T=\frac12\left[\frac12 V^{-1}CV-I\right].
\end{equation}
\end{subequations}
Consequently, in view of \eqref{P2g}, 
\begin{displaymath}
(I+T)\begin{bmatrix}0\\g\end{bmatrix}=T\begin{bmatrix}1\\ \Gamma Ph+\sigma\end{bmatrix}.
\end{displaymath}
Assuming that $I+T$ is nonsingular, for the moment as a technical condition to be more carefully investigated in the future, we define
\begin{equation}
\label{uUNP}
\begin{bmatrix}u&U\end{bmatrix}:=\begin{bmatrix}0&I_n\end{bmatrix}(I_{n+1}+T)^{-1}T,
\end{equation}
where for clarity we have added an index to each identity matrix to indicate dimension. Then 
\begin{equation}
\label{g2}
g=u+U\sigma +U\Gamma Ph,
\end{equation}
which has the same form as \eqref{g(P)}. In combination with \eqref{P2g} and \eqref{rhok}, this also yields
\begin{equation}
\label{arho2}
\begin{split}
&a=(I-U)(\Gamma Ph+\sigma)-u,\\
&\rho=\sqrt{1-h'Ph} ,
\end{split}
\end{equation} 
which is of the same form as \eqref{arho}.

We have thus demonstrated that the Covariance extension equation can be used also in this case after changing the definition of the interpolation parameters $(u,U)$. Hence the algorithms using homotopy continuation presented in \cite{BFL,BFLcdc05} could also be used here.

\section{CONCLUSIONS}\label{sec:conclusions}

In this paper we provide a new derivation of a nonstandard Riccati equation (CEE) for rational covariance extension that separates the part that only depends on the positivity and rationality constraints and the part that depends on the interpolation condition. In this way we see that the structure of CEE remains intact, and only certain interpolation parameters need to be modified when treating more general analytic interpolation problems with rationality constraints. 

It should be possible to generalize this framework to the MIMO case. This, together with some remaining numerical and technical issues, will the topic of a future paper.

\addtolength{\textheight}{-12cm}   



\section*{APPENDIX}

For ease of reference, we here reproduce the well-known {\em matrix inversion lemma}. 
Provided all inverses exist, the formula 
\begin{equation}
\label{app2:matrixinverstionlemma}
(A+BCD)^{-1}=A^{-1}-A^{-1}B(DA^{-1}B+C^{-1})^{-1}DA^{-1}
\end{equation}
holds for otherwise arbitrary matrices of compatible dimensions. This is seen by direct computation.


\end{document}